\documentclass[12 pt]{amsart}

\usepackage{amsmath,amsthm,amssymb,amscd}

\newcommand{\B}{\mathcal{B}}
\newcommand{\Z}{\mathbb{Z}}
\newcommand{\T}{\mathbb{T}}
\renewcommand{\P}{\mathcal{P}}
\newcommand{\Q}{\mathcal{Q}}
\newcommand{\U}{\mathcal{U}}

\renewcommand{\max}{\mathrm{max}}

\newcommand{\R}{\mathbb{R}}
\newcommand{\dom}{\mathrm{dom}}
\newcommand{\sd}{\mathrm{sd}}
\newcommand{\N}{\mathbb{N}}

\newcommand{\ind}{\mathrm{ind}}

\newtheorem{theorem}{Theorem}
\newtheorem{lemma}[theorem]{Lemma}
\newtheorem{corollary}[theorem]{Corollary}

\newtheorem*{claim}{Claim}

\theoremstyle{definition}
\newtheorem{definition}[theorem]{Definition}

\author{Marcin Sabok}\thanks{This research was supported by
  the MNiSW (Polish Ministry of Science and Higher
  Education) grant no N N201 418939 and by the Foundation
  for Polish Science}

\title{Extreme amenability of abelian $L_0$ groups}

\address{Marcin Sabok, University of Illinois at
  Urbana--Champaign, 1409 W. Green Street, Urbana, IL 61801,
  USA}

\address{Instytut Matematyczny Uniwersytetu
  Wroc\l awskiego, pl.  Grunwaldzki $2\slash 4$, $50$-$384$
  Wroc\l aw, Poland}

\address{Instytut Matematyczny Polskiej Akademii Nauk,
  ul. \'Sniadeckich $8$, $00$-$956$ Warszawa, Poland}

\email{sabok@math.uni.wroc.pl}


\begin{document}

\maketitle

\begin{abstract}
  We show that for any abelian topological group $G$ and
  arbitrary diffused submeasure $\mu$, every continuous
  action of $L_0(\mu,G)$ on a compact space has a fixed
  point. This generalizes earlier results of Herer and
  Christensen, Glasner, Furstenberg and Weiss, and Farah and
  Solecki. This also answers a question posed by Farah and
  Solecki. In particular, it implies that if $H$ is of the
  form $L_0(\mu,\R)$, then $H$ is extremely amenable if and
  only if $H$ has no nontrivial characters, which gives an
  evidence for an affirmative answer to a question of
  Pestov. The proof is based on estimates of chromatic
  numbers of certain graphs on $\mathbb{Z}^n$. It uses tools
  from algebraic topology and builds on the work of Farah
  and Solecki.
\end{abstract}

\section{Introduction}

A topological group is called \textit{extremely amenable} if
every its continuous action on a compact space has a fixed
point. This terminology comes from a characterization of
amenable groups saying that a (locally compact) group is
amenable if and only if every its continuous action on a
compact convext set in a locally convex vector topological
space has a fixed point. No nontrivial locally compact group
is extremely amenable since locally compact groups admit
free actions on compact spaces, by a theorem of
Veech. Still, however, many non locally compact groups prove
to be extremely amenable. These groups have received
considerable attention recently (see \cite{Pestov}) and the
following groups were proved to be extremely amenable: the
group $\textrm{Homeo}_+([0,1])$ of order-preserving
homeomorphisms of the unit interval (Pestov,
\cite{Pestov:Free.actions}), the group $U(\ell^2)$ of
unitary transformations of $\ell^2$ (Gromov and Milman,
\cite{Gromov.Milman}), the group
$\textrm{Aut}([0,1],\lambda)$ of measure-preserving
automorphisms of the Lebesgue space (Giordano and Pestov,
\cite{Giordano.Pestov}), the group
$\textrm{Iso}(\mathbb{U})$ of isometries of the Urysohn
space (Pestov, \cite{Pestov:Ramsey}) and many groups of the
form $\textrm{Aut}(M)$ for a large class of countable
structures $M$ (Kechris, Pestov and Todorcevic,
\cite{Kechris.Pestov.Todorcevic}). The methods of proofs in
the above cases vary from the concentration of measure
phenomena to structural Ramsey theory. On the other hand,
the earliest known examples of extremely amenable groups
were of the form $L_0(\mu,G)$ for a topological group $G$
and a submeasure $\mu$, see \cite{Herer.Christensen,
  Glasner}. 

Let us explain the notion of a submeasure and an $L_0$
group. Given an algebra $\B$ of subsets of a set $X$, we say
that a function $\mu:\B\rightarrow[0,\infty)$ is a
\textit{submeasure} if $\mu(U)\leq\mu(V)$ whenever
$U\subseteq V\subseteq X$ and $\mu(U\cup
V)\leq\mu(U)+\mu(V)$ for all $U,V\subseteq X$. A submeasure
$\mu$ is called a \textit{measure} if $\mu(U\cup
V)=\mu(U)+\mu(V)$ for disjoint $U,V\subseteq X$. A
submeasure $\mu$ is called \textit{diffused} if for every
$\varepsilon>0$ there is a finite cover $X=\bigcup_{i\leq
  n}X_i$ with $X_i\in\B$ and $\mu(X_i)\leq \varepsilon$.

Given a topological group $G$, the group $L_0(\mu,G)$ is the
group of all step $\B$-measurable functions from $X$ to $G$
with finite range, with the pointwise multiplication and the
topology of convergence in submeasure $\mu$.\footnote{ This
  construction appears in \cite{Hartman.Mycielski}. In
  \cite{Solecki.Farah} the authors take the metric
  completion in case $G$ is locally compact. Since extreme
  amenability is invariant under taking dense subgroups, the
  choice of the definition does not affect extreme
  amenability of $L_0$ groups.} This means that given $f\in
L_0(\mu,G)$, a neibourhood $V$ of the identity in $G$ and
$\varepsilon>0$, a basic neighbourhood of $f$ is given by
$$\{h\in L_0(\mu,G): \mu(\{x\in X: h(x)\notin V\cdot
f(x)\})<\varepsilon\}.$$ Groups of this form have been
studied extensively especially in case $G=\R$. We write
$L_0(\mu)$ for $L_0(\mu,\R)$. In general form they have been
used by Hartman and Mycielski \cite{Hartman.Mycielski} to
show that any topological group can be embedded into a
connnected topological group. This is why these groups are
sometimes called the Hartman--Mycielski extensions. In
\cite{Hartman.Mycielski} Hartman and Mycielski proved that
in case $\mu$ is the Lebesgue measure, $L_0(\mu,G)$ contains
$G$ as a closed subgroup and is path-connected and
locally-path connected. The same remains true if $\mu$ is a
diffused submeasure. Thus, the map $G\mapsto L_0(\mu,G)$ is
a convariant functor from the category of topological groups
to the category of connected topological groups. Later,
Keesling \cite{Keesling} showed that if $G$ is separable and
metrizable, then the Hartman--Mycielski extension
$L_0(\mu,G)$ is homeomorphic to $\ell^2$.

A \textit{character} of an abelian topological group is a
continuous homomorphism to the unit circle $\T$. Nikodym
\cite{Nikodym} studied the groups $L_0(\mu)$ and essentially
showed \cite[Pages 139--141]{Nikodym}, that $L_0(\mu)$ does
not have nontrivial characters if and only if $\mu$ is
diffused. The same argument shows that if $G$ is any abelian
group and $\mu$ is a diffused measure, then $L_0(\mu,G)$
does not have nontrivial characters. Note at this point that
if an abelian group has a nontrivial character, then it
admits a free action on the circle $\T$ and thus cannot be
extremely amenable. Therefore, if $\mu$ is not diffused and
$G$ is abelian, then $L_0(\mu,G)$ is not extremely
amenable. One of the major open problems in the field is a
question of Pestov whether for any abelian group $H$,
extreme amenability of $H$ is equivalent to the fact that
$H$ has no nontrivial characters.

A submeasure $\mu$ is \textit{pathological} if there is no
nonzero measure $\nu$ on $X$ with $\nu\leq \mu$. A
submeasure $\mu$ is \textit{exhaustive} if for every
disjoint family $A_n\in\B$ we have $\lim_n \mu(A_n)=0$. For
a recent construction by Talagrand (answering an old
question of Maharam) of an exhaustive pathological
submeasure see \cite{Talagrand}. Extreme amenability of the
groups $L_0(\mu,G)$ has been studied and proved by several
authors for various groups and submeasures, and the
following list summarizes the current state of
knowledge. The group $L_0(\mu,G)$ is extremely amenable in
each of the following cases:
\begin{itemize}
\item if $G=\R$ and $\mu$ is pathological (Herer and
  Christensen, \cite{Herer.Christensen}),
\item if $G=\T$ and $\mu$ is a diffused measure (Glasner
  \cite{Glasner} and, independently, Furstenberg and Weiss,
  unpublished),
\item if $G$ is amenable and $\mu$ is a diffused measure
  (Pestov, \cite{Pestov:Ramsey}),
\item if $G$ is compact solvable and $\mu$ is diffused
  (Farah and Solecki, \cite{Solecki.Farah})
\end{itemize}

There is a number of cases in which the answer was not
known. The situation was unclear even for the groups $\Z$
and $\R$. In \cite[Question 2]{Solecki.Farah} Farah and
Solecki ask for what submeasures are the groups
$L_0(\mu,\Z)$ and $L_0(\mu,\R)$ extremely amenable. The same
question appears also in Pestov's book \cite[Question
8]{Pestov.old}. In this paper we answer this question and
prove the following.

\begin{theorem}\label{main}
  For any abelian topological group $G$ and arbitrary
  diffused submeasure $\mu$ the group $L_0(\mu,G)$ is
  extremely amenable.
\end{theorem}

Theorem \ref{main} easily generalizes to solvable groups in
place of abelian, using \cite[Lemma 3.4]{Solecki.Farah}.
Combined with the theorem of Nikodym, Theorem \ref{main}
implies the following.

\begin{corollary}
  Let $\mu$ be a submeasure. The following are equivalent:
  \begin{itemize}
  \item $\mu$ is diffused,
  \item $L_0(\mu,G)$ is extremely amenable for every abelian
    $G$,
  \item $L_0(\mu,\Z)$ is extremely amenable.
  \end{itemize}
\end{corollary}

On the othe hand, Theorem \ref{main} provides an evidence
for an affirmative answer to the question of Pestov
mentioned above, as it implies the following.

\begin{corollary}
  Let $\mu$ be a submeasure. The following are equivalent:
  \begin{itemize}
  \item $\mu$ is diffused,
  \item $L_0(\mu)$ is extremely amenable,
  \item $L_0(\mu)$ has no nontrivial characters.
  \end{itemize}
\end{corollary}

The proof of Theorem \ref{main} uses methods of algebraic
topology and the Borsuk--Ulam theorem. It builds on the work
of Farah and Solecki \cite{Solecki.Farah}, who discovered
that algebraic topological methods can be applied to study
extremely amenable groups. Farah and Solecki obtained a new
Ramsey-type theorem, whose proof was based on a construction
of a family of simplicial complexes and an application of
the Borsuk--Ulam theorem. In this paper, rather than proving
a Ramsey-type result, we show a connection of extreme
amenability of abelian $L_0$ groups with the existence of
finite bounds on chromatic numbers of certain graphs on
$\Z^n$. We use methods of algebraic topology to establish
bounds from below on these chromatic numbers. The
application of the Borsuk--Ulam theorem is based on the
ideas of Farah and Solecki but it employs different
(smaller) simplicial complexes.

On the other hand, it is worth noting that the methods of
algebraic topology have been also used by Lov\'asz
\cite{Lovasz} for showing lower bounds on chromatic numbers
of Kneser's graphs and by Matou\v sek \cite{Matousek:PAMS}
for the Kneser hypergraphs.

This paper is organized as follows. In Section
\ref{sec:chromatic} we connect extreme amenability of
abelian $L_0$ groups with colorings of graphs. In Section
\ref{sec:complexes} we construct a family of siplicial
complexes used later in the proof of Theorem. The main
result establishing the bounds on the chromatic numbers is
proved in Section \ref{sec:bounds}.

\textbf{Acknowledgement}. I would like to thank S\l awomir
Solecki for valuable discussions during my stay in
Urbana--Champaign and for many helpful comments.

\section{Extreme amenability and chromatic
  numbers}\label{sec:chromatic}

Now we define the basic object of this paper, the graphs
associated to a submeasure and a real $\varepsilon>0$. We
call a finite partition $\P$ of $X$ a \textit{measurable
  partition} if all elements of $\P$ belong to $\B$.

\begin{definition}
  Let $\mu$ be a submeasure on a set $X$, let
  $\varepsilon>0$ and let $\P$ be a finite measurable
  partition of $X$. We define the (undirected) graph
  $\Gamma^\P_\varepsilon(\mu)$ as follows. The set of nodes
  is $\Z^{\P}$ and two nodes $k=(k_A:A\in\P)$ and
  $l=(l_A:A\in\P)$ are connected with an edge
  if $$\mu(\bigcup\{A\in\P: k_A\not=l_A+1\})<\varepsilon.$$
\end{definition}

A \textit{coloring} of a graph is a function defined on the
set of its vertices such that no two vertices connected with
an edge are assigned the same value.  Given a graph $\Gamma$
we write $\chi(\Gamma)$ for its \textit{chromatic number},
i.e. the least number of colors in a coloring of the graph
$\Gamma$. The following lemma establishes the main
connection between extreme amenability of abelian $L_0$
groups and chromatic numbers of the graphs
$\Gamma^\P_\varepsilon(\mu)$.

\begin{lemma}\label{connection}
  Let $\mu$ be a submeasure on a set $X$. The following are
  equivalent:
  \begin{itemize}
  \item[(i)] the group $L_0(\mu,G)$ is extremely amenable
    for every abelian topological group $G$,
  \item[(ii)] for every $\varepsilon>0$ there is no finite
    bound on the chromatic
    numbers $$\chi(\Gamma^\P_\varepsilon(\mu)),$$ where $\P$
    runs over all finite measurable partitions of $X$.
  \end{itemize}
\end{lemma}
\begin{proof}
  (i)$\Rightarrow$(ii) Suppose that for some $\varepsilon$
  there is a finite bound $d$ on the chromatic numbers
  $\chi(\Gamma^\P_\varepsilon(\mu))$. It is enough to show
  that the group $L_0(\mu,\Z)$ is not extremely
  amenable. For each measurable partition $\P$ of $X$ pick a
  coloring $c_\P:\Z^\P\rightarrow\{1,\ldots,d\}$ of
  $\Gamma^\P_\varepsilon(\mu)$ into $d$ many colors. Fix an
  ultrafilter $\U$ on the set of all finite measurable
  partitions of $X$. For every step $\B$-measurable function
  $f:X\rightarrow\Z$ with finite range let $\P_f$ be the
  partition induced by $f$, i.e the one
  $\{f^{-1}(\{k\}):k\in\Z\}$. For every partition $\P$
  refining $\P_f$ let $f_\P$ be the element of $\Z^\P$ such
  that $f_\P(i)=k$ if and only if the $i$-th element of $\P$
  is contained in $f^{-1}(\{k\})$. Now let $c$ be a coloring
  of $L_0(\mu,\Z)$ defined as the limit over the
  ultrafilter: $$c(f)=\lim_\U c_\P(f_\P).$$ Let $X_i=\{f\in
  L_0(\mu,\Z): c(f)=i\}$ for each $i\in\{1,\ldots,d\}$ and
  note that the sets $X_i$ cover the group
  $L_0(\mu,\Z)$. Let $\bar 1\in L_0(\mu,\Z)$ be the constant
  function with value $1$. Note that by the definition of
  the graphs $\Gamma^\P_\varepsilon(\mu)$ we have that for
  each $i\in\{1,\ldots,d\}$ $$(\bar
  1+V_\varepsilon)\cap(X_i-X_i)=\emptyset,$$ where
  $V_\varepsilon=\{f\in L_0(\mu,\Z): \mu\{x\in X:
  f(x)\not=0\}<\varepsilon\}$ is the basic neighbourhood of
  the identity in $L_0(\mu,\Z)$. This implies that $\bar
  1\notin\overline{X_i-X_i}$ for each $i\leq d$, so
  $L_0(\mu,\Z)$ is not extremely amenable by \cite[Theorem
  3.4.9]{Pestov}.

  (ii)$\Rightarrow$(i) Let $G$ be an abelian topological
  group and suppose that $L_0(\mu,G)$ is not extremely
  amenable. By \cite[Theorem 3.4.9]{Pestov}, there is a left
  syndetic set $S\subseteq L_0(\mu,G)$ such that
  $G\not=\overline{S- S}$. Write $e$ for the neutral element
  of $G$. Assume that there are $d$ many left translates
  $S_1,\ldots,S_d$ of $S$ which cover $G$ and an element $f$
  of $L_0(\mu,G)$ such that $f\notin\overline{S-S}$. This
  means that there is a neighborhood $W$ of the identity in
  $G$ and an $\varepsilon>0$ such that $$(f+W_\varepsilon)\
  \cap\ (S-S)=\emptyset$$ where $W_\varepsilon=\{h\in
  L_0(\mu,G): \mu(\{x\in [0,1]: h(x)\notin
  W\})<\varepsilon\}$. Let $V_\varepsilon=\{h\in L_0(\mu,G):
  \mu(\{x\in [0,1]: h(x)\not=e\})<\varepsilon\}$ and note
  that $V_\varepsilon\subseteq W_\varepsilon$, so $(f+
  V_\varepsilon)\ \cap\ (S-S)=\emptyset$.

  Let $\P_f$ be the partition generated by $f$,
  i.e. $\P_f=\{f^{-1}(\{g\}):g\in G\}$.  Now we will
  construct $d$-colorings of the graphs
  $\Gamma^{\P}_\varepsilon(\mu)$ for all $\P$ refining
  $\P_f$. Since we have
  $\chi(\Gamma^{\P_0}_\varepsilon(\mu))\leq\chi(\Gamma^{\P_1}_\varepsilon(\mu))$
  whenever $\P_1$ refines $\P_0$, this will end the proof.

  Let $\P$ be a measurable partition of $X$ refining $\P_f$
  and let $k\in\Z^{\P}$. For each $A\in\P$ let $i_A\in\Z$ be
  such that $A\subseteq f^{-1}(\{i_A\})$ and let $g_k\in
  L_0(\mu,G)$ be such that $$g_k(x)=i_A f(x)\quad\mbox{if
  }x\in A.$$ Put $$c(k)=i\quad\mbox{if and only if}\quad
  g_k\in S_i.$$ We claim that $c$ defines a coloring of the
  graph $\Gamma^{\P}_\varepsilon(\mu)$. Indeed, if
  $k,l\in\Z^\P$ have the same color and are connected with
  an edge, then $\mu(\bigcup\{A\in\P:
  k(A)\not=l(A)+1\})<\varepsilon$, so $$\mu(\{x\in X:
  g_k(x)-g_l(x)\not=f(x)\})<\varepsilon$$ and hence $(f+
  V_\varepsilon)\cap (S- S)=(f+ V_\varepsilon)\cap (S_i-
  S_i)\not=\emptyset$.

\end{proof}

\section{Simplicial complexes}\label{sec:complexes}

Here we gather the definitions and constructions of
simplicial complexes used in the proof o Theorem \ref{graph}
in the next section.

Given a finite set $V$, a \textit{symplicial complex} with
vertex set $V$ is a finite family of subsets of $V$, closed
under taking subsets. Elements of this family are called
\textit{simplices} of the simplicial complex. Given two
simplicial complexes $K$ and $L$ we say that a map
$f:K\rightarrow L$ is \textit{simplicial} if $\{f(v):v\in
F\}\in L$ whenever $F\in K$. Given a simplicial complex $K$
with vertex set $V$ and an action of a group $H$ on $V$, we
say that $K$ is an \textit{$H$-complex} if for every $F\in
K$ and $h\in H$ we have $\{h(v):v\in F\}\in K$. Given two
$H$-complexes $K$ and $L$ and a simplicial map
$f:K\rightarrow L$ we say that $f$ is
\textit{$H$-equivariant} and write $f:K\xrightarrow{H} L$ if
$f(h(v))=h f(v)$ for each $v\in V$ and $h\in H$.  Given two
simplicial complexes $K$ and $L$ with disjoint sets of
vertices, we define their \textit{join} and denote it by
$K*L$ as the simplicial complex $$\{F\cup G: F\in K,G\in
L\}.$$ If $K$ is a simplicial complex, then its
\textit{barycentric subdivision}, denoted by $\sd(K)$, is
the simplicial complex with the vertex set $K$ defined
as $$\{C\subseteq K: C\not=\emptyset\mbox{ and }C\mbox{ is a
  chain}\}.$$ By $\sd^k(K)$ we denote the $k$-fold iteration
of the barycentric subdivision. If $K$ is an $H$-complex,
then $\sd(K)$ becomes an $H$-complex too, in the natural
way. Given a simplicial complex $K$, we write $||K||$ for
its \textit{geometric realization} (see \cite[Page
537]{Hatcher} or \cite[Page 14]{Matousek}). Note that if a
group $H$ acts on $K$, then this action can be naturally
extended to an action on $||K||$.

Given two simplicial complexes $K$ and $L$ and a simplicial
map $f:K\rightarrow L$, there is a natural simplicial map
$\sd(f):\sd(K)\rightarrow \sd(L)$. If $K$ and $L$ are
$H$-complexes and and $f$ is $H$-equivariant, then $\sd(f)$
is $H$-equivariant too.

Given a prime number $p$ we write $\Z\slash p$ for the
cyclic group of rank $p$ and $+_p$ for the addition modulo
$p$.

\begin{definition}
  Given $n,p\in\N$ and a partial function
  $f:\{1,\ldots,n\}\rightarrow\{1,\ldots,p\}$ a
  \textit{component interval} of $f$ is any maximal interval
  $I\subseteq\{1,\ldots,n\}$ such that $f$ is constant on
  $I\cap\dom(f)$. Note that the component intervals of $f$
  are pairwise disjoint and cover $\dom(f)$. Given a
  nonempty partial function
  $f:\{1,\ldots,n\}\rightarrow\{1,\ldots,p\}$, the
  \textit{last component interval of} $f$ is the one
  containing $\max(\dom(f))$.
\end{definition}

Let $l,n$ be natural numbers with $n\geq l$ and let $p$ be a
prime number. Define $V^{n,l}_p$ as the set of nonempty
partial functions $f$ from $\{1,\ldots, n\}$ to $\Z\slash p$
which have the following properties:
\begin{itemize}
\item[(i)] $n-|\dom(f)|\leq l$,
\item[(ii)] the number of component intervals of $f$ is at
  most $l$.
\end{itemize}

Let $K^{n,l}_p$ be the simplicial complex whose vertices set
is $V^{n,l}_p$ and whose simplices are those subsets of
$V^{n,l}_p$ which form a chain with respect to
inclusion. There is a natural action of $\Z\slash p$ on
$K^{n,l}_p$ defined as follows. If $f\in V^{n,l}_p$ and
$k\in\Z\slash p$, then $k+f\in V^{k,l}_p$ is such that
$\dom(k+f)=\dom(f)$ and $(k+f)(i)=k+f(i)$ for each
$i\in\dom(f)$. 

The following lemma is a variant of an analogous
statement proved by Farah and Solecki \cite[Page
4]{Solecki.Farah} for the complexes used in their proof.

\begin{lemma}\label{simplicial}
  There is a simplicial
  map $$s:\sd(K^{n,l}_p)\xrightarrow{\Z\slash p}
  K^{n+1,l}_p.$$
\end{lemma}
\begin{proof}
  Every element of $\sd(K^{n,l}_p)$ is a chain of partial
  functions $\mathcal{F}=\{f_1\subseteq\ldots\subseteq
  f_m\}$. The way we assign an element of $K^{n+1,l}_p$ to
  $\mathcal{F}$ will depend on whether $m=1$ or $m>1$.

  If $\mathcal{F}$ is a nontrivial chain, i.e. $f_1\not=
  f_m$, then note that $n-|\dom(f_m)|\leq l-1$. This implies
  that $f_m\in K^{n+1,l}_p$, as $n+1-|\dom(f_m)|\leq l$ and
  the condition (ii)) is satisfied for $f_m$ since
  $\mathcal{F}\in\sd(K^{n,l}_p)$. Put $s(\mathcal{F})=f_m$
  in this case.

  If $\mathcal{F}$ is a trivial chain, i.e.
  $\mathcal{F}=\{f_1\}$, then let $q$ be the value assumed
  at the last component interval of $f_1$. Put
  $s(\mathcal{F})=f_1\cup\{(n+1,q))\}$. Now
  $s(\mathcal{F})\in K^{n+1,l}_p$ as the number of component
  intervals of $s(\mathcal{F})$ is the same as that of
  $f_1$, and (i) is satisfied as
  $n+1-|\dom(s(\mathcal(F)))|=n-|\dom(f_1)|$.

  It is clear that so defined $s$ preserves the action of
  $\Z\slash p$.
\end{proof}

Let $S^{l+1}_p$ be the complex whose vectices are partial
functions $f$ from $\{1,\ldots l+1\}$ to $\Z\slash p$ with
$|\dom(f)|=1$ and let the simplices of $S^{l+1}_p$ be those
subsets of its vertices whose union forms a partial function
from $\{1,\ldots l+1\}$ to $\Z\slash p$. In other words,
$S^{l+1}_p$ is the join $(\Z\slash p)^{*(l+1)}$.

\begin{lemma}\label{connected}
  $S^{l+1}_p$ is $l$-connected.
\end{lemma}
\begin{proof}
  This follows from the fact that if $K$ is $k$-connected
  and $L$ is $l$-connected, then $K*L$ is $k+l+2$-connected
  \cite[Proposition 4.4.3]{Matousek}.
\end{proof}

Note that $V^{l+1,l+1}_p$ consists of all nonempty partial
functions from $\{1,\ldots l+1\}$ to $\Z\slash p$, as the
conditions (iii) and (iv) are empty in this case. Thus, the
identity is a simplicial map
$$i:S^{l+1}_p\xrightarrow{\Z\slash p}  K^{l+1,l+1}_p,$$ which
obviously preserves the action of $\Z\slash p$.

\section{Bounds on chromatic numbers}\label{sec:bounds}

Given a diffused submeasure $\mu$, for each $\delta>0$ there
is a finite covering of $X$ with sets in $\B$ with
submeasure less than $\delta$. Let $k_\mu(\delta)$ be the
minimal number of elements in such a covering.  Given
$\varepsilon>0$ consider the function
$k_\varepsilon^\mu:\N\rightarrow\N$ defined as
$k_\varepsilon^\mu(d)=k_\mu(\varepsilon\slash 4d)$ and for
each $d\in\N$
let $$\Q^\varepsilon_d=\{I^\varepsilon_1,\ldots,I^\varepsilon_{k_\varepsilon^\mu(d)}\}$$
be a measurable partition of $X$ into $k_\varepsilon^\mu(d)$
many sets of submeasure less than $\varepsilon\slash
4d$. Note that $k_\varepsilon^\mu$ is increasing and
\begin{equation}\label{increasing}
  k_\varepsilon^\mu(d)\geq d\mu(X)(4\slash\varepsilon).
\end{equation}
Let $F_\varepsilon^\mu:\N\rightarrow\N$ be any increasing
function such that
\begin{equation}\label{inverse}
  k_\varepsilon^\mu\circ F_\varepsilon^\mu(m)=m
\end{equation}
for each $m\in\N$. Such a function exists since
$k_\varepsilon^\mu$ is increasing and unbounded by
(\ref{increasing}). Moreover,
\begin{equation}\label{limit}
  \lim_{m\rightarrow\infty}F_\varepsilon^\mu(m)=\infty.  
\end{equation}
The rate of growth of $F_\varepsilon^\mu$ may be very low if
$\mu$ is not a measure (for example, for the Hausdorff
submeasures constructed in \cite{Solecki.Farah}). On the
other hand, if $\mu$ is a measure, then $F^\mu_\varepsilon$
is linear. However, for proving Theorem \ref{main} we only
need the fact that $F^\varepsilon_\mu$ diverges to infinity.

For $\varepsilon>0$ write
$C^\mu_\varepsilon=\sqrt[3]{\mu(X)^2\slash
  16\varepsilon}$. For $n\in\N$ and $\varepsilon>0$ let
$\P^\varepsilon_n(\mu)$ be a finite measurable partition of
$X$ with, say $k^\varepsilon_n(\mu)$, many sets in $\B$ of
submeasure less than $1\slash n$ which refines all
$\Q^\varepsilon_d(\mu)$ for $d\leq
F^\mu_\varepsilon(C^\mu_\varepsilon\sqrt[3]{n})$. Write
$\Gamma^n_\varepsilon(\mu)$ for the graph
$\Gamma^{\P^\varepsilon_n(\mu)}_\varepsilon(\mu)$.

Here is the main result which, composed with Lemma
\ref{connection} and (\ref{limit}), implies Theorem
\ref{main}.

\begin{theorem}\label{graph}
  Let $\mu$ be a diffused submeasure. Given $\varepsilon>0$
  we have $$\chi(\Gamma^n_\varepsilon(\mu))\geq
  F_\varepsilon^\mu(C^\mu_\varepsilon\sqrt[3]{n}).$$
\end{theorem}
\begin{proof}
  Fix $n\in\N$. Write $k_n$ for $k^\varepsilon_n(\mu)$ and
  $\mu_n$ for the submeasure on $\{1,\ldots,k_n\}$ induced
  by $\mu$, i.e $\mu_n(B)=\mu(\bigcup\{A_i:i\in B\})$, where
  $\{A_1,\ldots,A_{k_n}\}$ is an enumeration of $\P_n(\mu)$.

  Suppose that
  $d<K^\varepsilon_\mu(C^\mu_\varepsilon\sqrt[3]{n})$ and
  there exists a coloring
  $c:\Z^{k_n}\rightarrow\{1,\ldots,d\}$ of the graph
  $\Gamma^n_\varepsilon(\mu)$. Let $k=k(\varepsilon\slash
  4d)$. By the assumption on $\P^\varepsilon_n(\mu)$, we can
  assume (possibly rearranging the numbers $1,\ldots,k_n$)
  that there are consequtive intervals $I_1,\ldots,I_k$
  covering $\{1,\ldots,k_n\}$ each of submeasure $\mu_n$
  less than $1\slash 4d$. By the Bertrand postulate
  \cite[Chapter 10]{Sierpinski.Numbers}, pick a prime number
  $p$ with
  \begin{equation}
    \label{eq:eqq0}
    k(d+1)<p<2k(d+1)
  \end{equation}
  Since $p>k(d+1)$, we have that
  \begin{align*}
    \frac{k}{p}&<\frac{1}{d+1}=1-\frac{d}{d+1},\\
    1-\frac{k}{p}&>\frac{d}{d+1},\\
    (1-\frac{k}{p})(d+1)&>d,
  \end{align*}
  which implies that
  \begin{equation}
    \label{eq:eqq1}
    (p-k)(d+1)>dp    .
  \end{equation}
  Put $l=dp$. 
  \begin{claim}
    $$\frac{l+1}{n}<\frac{\varepsilon}{8}$$
  \end{claim}
  \begin{proof}
    Otherwise, we have $dp=l\geq n\,\varepsilon\slash
    8$ and since $k=k(\varepsilon\slash 4d)$, we have that
    \begin{equation}\label{eq:k}
      \frac{4d}{\varepsilon}\mu(X)\leq k
    \end{equation}
    and
    \begin{equation}\label{eq:epsilon}
      d+1\leq 4d=\frac{\varepsilon}{\mu(X)}(\frac{4d}{\varepsilon}\mu(X))\leq\frac{\varepsilon}{\mu(X)}k.
    \end{equation}
    By (\ref{eq:eqq0}) and the assumption that the Claim
    does not hold, we have
    \begin{equation*}
      \frac{n\varepsilon}{8}\leq dp\leq 2k(d+1)d< 2k(d+1)^2,\\
    \end{equation*}
    and by (\ref{eq:epsilon}) we get
    \begin{equation}
      \frac{n\varepsilon}{16}< k(d+1)^2\leq \frac{\varepsilon^2}{\mu(X)^2} k^3,
    \end{equation}
    which gives that
    $k=k^\varepsilon_\mu(d)>C^\mu_\varepsilon\sqrt[3]{n}$. Since
    $F^\varepsilon_\mu$ is increasing, (\ref{inverse})
    implies that $d\geq F^\varepsilon_\mu(C\sqrt[3]{n})$
    contrary to the assumption. This ends the proof of the Claim.  
  \end{proof}
  
  Write $l_n=k_n-l-1$. By Lemma \ref{simplicial}, we get a
  simplicial map
  $$\sd^{l_n}(K^{l+1,l+1}_p)\xrightarrow{\Z\slash p} K^{k_n,l+1}_p$$
  and precomposing it with $\sd(i)^{l_n}$ we get
  $$s:\sd^{l_n}(S^{l+1}_p)\xrightarrow{\Z\slash p}K^{k_n,l+1}_p.$$

  For each $f\in V^{k_n,l+1}_p$ write $\bar f\in(\Z\slash
  p)^{k_n}$ for the function which is equal to $f$ on its
  domain and $0$ elsewhere. Let $\bar
  c:V^{k_n,l+1}_p\rightarrow\{1,\ldots,d\}$ be defined so
  that $\bar c(f)=c(\bar f)$. Note that $\bar c$ induces a
  map $c':V^{k_n,l+1}_p\rightarrow \R^d$ so that
  $c'(f)=\mathrm{e}_{\bar c(f)}$, where
  $\{\mathrm{e}_1,\ldots,\mathrm{e}_d\}$ are the standard
  basis vectors in $\R^d$. This in turn, extends to a
  continuous
  map $$||c'||:||K^{k_n,l+1}_p||\rightarrow\R^d,$$ which is
  the affine extension of $c'$.

  Write $c''$ for $s\circ
  c':\sd^{l_n}(S^{l+1}_p)\rightarrow\R^d$ and let
  $$||c''||:||\sd^{l_n}(S^{l+1}_p)||\rightarrow\R^d$$ be its
  affine extension. Let also
  $||s||:||S^{l+1}_p||\rightarrow||K^{k_n,l+1}_p||$ be the
  affine extension of $s$. Note that
  $||c''||=||s||\circ||c'||$.

  Since $S^{l+1}_p$ is $l$-connected by Lemma
  \ref{connected}, so is $\sd^{l_n}(S^{l+1}_p)$ and hence
  its $\Z\slash p$-index (see \cite[Definition
  6.2.3]{Matousek} for the definition)
  satisfies $$\ind_{\Z\slash p}(\sd^{l_n}(S^{l+1}_p))\geq
  l+1$$ by \cite[Proposition 6.2.4]{Matousek}. Now,
  $l+1>d(p-1)$, so by the Borsuk--Ulam theorem \cite[Theorem
  6.3.3]{Matousek}, there is a point
  $x_0\in||\sd^{l_n}(S^{l+1}_p)||$ such that $||c''||$ is
  constant on the $\Z\slash p$-orbit of $x_0$. Since $||s||$
  preserves the action of $\Z\slash p$, we get that $||c'||$
  is constant on the $\Z\slash p$-orbit of $||s||(x_0)$. The
  point $||s||(x_0)$ lies in a maximal simplex in
  $K^{k_n,l+1}_p$, which is of the
  form $$h_{l+1}\subseteq\ldots\subseteq h_0.$$ Pick $i_0<d$
  such that the $i_0$-th coordinate of $||s||(x_0)$ is
  nonzero. As the function $||c'||$ is the affine extension
  of $c'$, there is $0\leq m\leq l+1$ such that $\bar
  c(h_m)=i_0$. Moreover, if $q\in\Z\slash p$ is any number,
  then $||s||(x_0)+q$ lies in the maximal
  simplex $$h_{l+1}+q\subseteq\ldots\subseteq h_0+q$$ and
  since still the $i_0$-th coordinate of $||s||(x_0)+q$ is
  nonzero, there is $0\leq m_q\leq l+1$ such that $\bar
  c(h_{m_q}+q)=i_0$.

  Put $h=h_{m_0}$. Consider the set $A_h=\{j\in\Z\slash p:$
  there is a component interval $J$ of $h$ such that $h$ is
  equal to $j$ on $J$ and $J$ intersects at least two of the
  intervals $I_1,\ldots, I_k\}$. Note that since the
  intervals $I_1,\ldots,I_k$ are disjoint and consequtive,
  we have $|A_h|\leq k-1$. Let
  $B_h=\{0,\ldots,p-1\}\setminus A_h$ and note that
  $|B_h|\geq p-k+1$. For each $j\in B_h$ the preimage of
  ${j}$ by $h$ is a disjoint union of component intervals,
  each of which is contained in at most one of the intervals
  $I_1,\ldots,I_k$. Note that (\ref{eq:eqq1}) implies that
  $$(p-k+1)(d+1)>l+1.$$ Since the number of the component
  intervals of $h$ is at most $l+1$, by the pigeonhole
  principle we get that one of the elements of $B_h$, say
  $q_0$, must have at most $d$ many component intervals in
  its preimage. Since each $I_1\ldots, I_k$ has submeasure
  at most $\varepsilon\slash 4d$, we get that

  \begin{equation}\label{preimage}
    \mu_n(h^{-1}(\{q_0\}))\leq
    d\cdot\frac{\varepsilon}{4d}=\frac{\varepsilon}{4}.
  \end{equation}

  Let $f=h_{m_{p-1-q_0}}+_p(p-1-q_0)$ and
  $g=h_{m_{p-q_0}}+_{p}(p-q_0)$.  Note that by the way we
  have chosen the numbers $m_q$, we have $c(\bar f)=c(\bar
  g)=i_0$. We will show that
  \begin{equation}\label{contradiction}
    \mu_n(\{i\leq k_n:\bar f(i)+1\not=\bar g(i)\})<\varepsilon,
  \end{equation}
  which implies that $\bar f$ and $\bar g$ are connected
  with an edge in $\Gamma^n_\varepsilon(\mu)$. This will
  give a contradiction and end the proof.

  Note that $h$ and $h_{m_{p-1-q_0}}$ as well as $h$ and
  $h_{m_{p-q_0}}$ differ at at most $l+1$ many points in
  $\{1,\ldots,k_n\}$. Since each point has submeasure
  $\mu_n$ at most $1\slash n$, by the Claim we get that
  \begin{align}
    \mu_n(\{i\leq k_n:h_{m_{p-1-q_0}}(i)\not=
    h(i)\})\leq\frac{\varepsilon}{8}\label{one}
    \\\mu_n(\{i\leq k_n:h_{m_{p-q_0}}(i)\not=
    h(i)\})\leq\frac{\varepsilon}{8}\label{two}
 \end{align}

 By (\ref{preimage}), (\ref{one}) and the Claim we have that
  \begin{align*}\label{arithmetic}
    &\mu_n(\{i\leq k_n:\bar f(i)+_p1\not= \bar
    f(i)+1\})=\mu_n(\bar
    f^{-1}(\{p-1\}))\tag{13}\\&\leq\mu_n(f^{-1}(\{p-1\}))+\varepsilon\slash
    8\leq\mu_n({h_{m_{p-1-q_0}}}^{-1}(\{q_0\}))+\varepsilon\slash
    8\\&\leq \mu_n(h^{-1}(\{q_0\}))+\varepsilon\slash
    8+\varepsilon\slash 8\leq \varepsilon\slash
    8+\varepsilon\slash 4.
  \end{align*}

  On the other hand, by (\ref{one}), (\ref{two}) and the Claim
  \begin{align*}
    &\mu_n(\{i\leq k_n:\bar f(i)+_p 1\not=
    h(i)+_p(p-q_0)\})\\&=\mu_n(\{i\leq k_n:\bar f(i)\not=
    h(i)+_p(p-1-q_0)\})\\&\leq\mu_n(\{i\leq k_n:f(i)\not=
    h(i)+_p(p-1-q_0)\})+\varepsilon\slash 8\\&=\mu_n(\{i\leq
    k_n:h_{m_{p-1-q_0}}(i)+_p(p-1-q_0)\not=
    h(i)+_p(p-1-q_0)\})\\&+\varepsilon\slash 8=\mu_n(\{i\leq
    k_n:h_{m_{p-1-q_0}}(i)\not= h(i)\})+\varepsilon\slash
    8\leq\varepsilon\slash 4
  \end{align*}
  and
  \begin{align*}
    &\mu_n(\{i\leq k_n:\bar g(i)\not=
    h(i)+_p(p-q_0)\})\\&\leq\mu_n(\{i\leq k_n:g(i)\not=
    h(i)+_p(p-q_0)\})+\varepsilon\slash 8\\&=\mu_n(\{i\leq
    k_n:h_{m_{p-q_0}}(i)+_p(p-q_0)\not=
    h(i)+_p(p-q_0)\})+\varepsilon\slash 8\\&=\mu_n(\{i\leq
    k_n:h_{m_{p-q_0}}(i)\not= h(i)\})+\varepsilon\slash
    8\leq\varepsilon\slash 4.
  \end{align*}
  So, together with (\ref{arithmetic}), this gives
  (\ref{contradiction}), as needed. This ends the proof.
 \end{proof}

\bibliographystyle{plain}
\bibliography{refs}

\end{document}